\documentclass{siamonline250211}

\usepackage[T1]{fontenc}
\usepackage{amsfonts,amssymb,mathtools,mathrsfs,bm}
\usepackage{booktabs,array}
\usepackage{enumitem}
\allowdisplaybreaks[3]
\newsiamremark{remark}{Remark}
\crefname{remark}{Remark}{Remarks}
\renewcommand{\keywordsname}{Keywords}
\renewcommand{\MSCcodesname}{MSC(2020)}

\newcommand{\CC}{\mathbb C}

\newcommand{\PP}{\mathbb P}
\newcommand{\Gr}{\operatorname{Gr}}
\newcommand{\Pl}{\operatorname{Pl}}
\newcommand{\Hom}{\operatorname{Hom}}
\newcommand{\Span}{\operatorname{span}}
\newcommand{\GL}{\operatorname{GL}}
\newcommand{\Sym}{\operatorname{Sym}}
\newcommand{\rank}{\operatorname{rank}}
\newcommand{\Sing}{\operatorname{Sing}}
\newcommand{\cK}{\mathcal K}
\newcommand{\cC}{\mathcal C}
\newcommand{\cU}{\mathcal U}
\newcommand{\dd}{\mathrm d}

\title{Grassmann--Pl\"ucker Parametrization of Convolutional Filter Subspaces:\
Regularity and Closed Embeddings}
\author{H.~Yuan\thanks{Zhili College, Tsinghua University, Beijing, 100084,
P.~R.~China. Email address: \email{yuanhy24@mails.tsinghua.edu.cn}.}
\and H.~Zuo \thanks{Department of Mathematical Sciences, Tsinghua University,
Beijing, 100084, P.~R.~China. Email address: \email{hqzuo@mail.tsinghua.edu.cn}.\\
Corresponding author: Huaiqing Zuo.\\
Data availability: Data sharing not applicable to this article as no datasets were generated or analysed during the current study.
}
}

\headers{Grassmann--Pl\"ucker Filter Subspaces}{H.~Yuan and H.~Zuo}

\begin{document}
\maketitle

\begin{abstract}
We propose a geometric parametrization of the filters in a single convolutional
layer: the parameter is no longer an ordered family of filter vectors, but a
fixed-dimensional subspace of the filter space. For one-dimensional finite-stride
convolution, the filter-to-convolution-operator correspondence gives an injective
linear map $\cC:\cK\to H$. This map sends filter subspaces in $\Gr(q,\cK)$ to
operator subspaces in $\Gr(q,H)$; composing it with the Pl\"ucker embedding yields
a projective parametrization $\Phi:\Gr(q,\cK)\to\PP(\wedge^qH)$. Using
$T_U\Gr(q,\cK)\cong\Hom(U,\cK/U)$, we compute the differential of
$\Gamma_{\cC}$ and prove that the differential of $\Phi$ is injective at every
point. We then use the vanishing equations for Pl\"ucker coordinates and the
standard affine coordinates on a Grassmannian to prove that the sub-Grassmannian
$\Gr(q,\cC(\cK))\hookrightarrow\Gr(q,H)$ is a closed embedding, and hence that
$\Phi$ is a closed embedding. Consequently, the parameter space is isomorphic
to its projective image, the parametrization is finite and birational onto its
image, every fiber is a singleton, and the resulting projective neural variety
is smooth. For the first nontrivial case $k=4$ and $q=2$, we also use Singular
to eliminate the source Pl\"ucker coordinates and recover the image ideal
directly, checking its dimension, degree, chart rank, and smoothness. This
symbolic computation is a low-dimensional illustration rather than a substitute
for the general proof. Finally, we discuss possible connections with filter
redundancy and low-rank convolution, while distinguishing the geometric results
proved here from application proposals that still require numerical validation.
\end{abstract}



\section{Introduction}

We study a projective parametrization of a class of subspaces of single-layer
convolution operators. A traditional single-filter parameter is a vector in the
filter space, whereas here the parameter is generalized to a fixed-dimensional
subspace of that space. Let $\cK$ be the filter space and let $H$ be the space of
linear operators from the input space to the output space. Convolution gives a
linear map
\[
  \cC:\cK\longrightarrow H.
\]
For $U\in\Gr(q,\cK)$, its image $\cC(U)$ is a linear subspace of $H$. If $\cC$
is injective, then $\dim\cC(U)=q$, and hence there is a map
\[
  \Gamma_{\cC}:\Gr(q,\cK)\longrightarrow\Gr(q,H),
  \qquad U\longmapsto\cC(U).
\]
Composing it with the Pl\"ucker embedding gives the parametrization studied in
this paper:
\[
  \Phi=\Pl_H\circ\Gamma_{\cC}:
  \Gr(q,\cK)\longrightarrow\PP(\wedge^qH).
\]

The body of the paper addresses three questions in order. First, we construct
$\cC$ directly from the finite-stride convolution formula and prove that $\cC$
is injective. Second, we compute the differentials of $\Gamma_{\cC}$ and $\Phi$
and prove that the differential of $\Phi$ is injective at every point. Third,
setting $W=\cC(\cK)$, we use Pl\"ucker coordinates to prove that the natural
inclusion
\[
  \Gr(q,W)\longrightarrow\Gr(q,H)
\]
is a closed embedding, from which the closed-embedding property of $\Phi$
follows. Finiteness, birationality, uniqueness of fibers, and smoothness of the
image are then consequences of the closed-embedding theorem.
In addition, for the concrete choice $k=4,q=2,d=5,d'=2,s=1$, we use Singular
to eliminate the six source Pl\"ucker coordinates from the graph ideal, obtain
the homogeneous ideal of the projective image, and compare the computation
term by term with the main theorem.

The main result is stated explicitly as follows.

\begin{theorem}[Main theorem]\label{thm:main}
Let $\cK=\CC^k$, let $H=\Hom(E,F)$, and let
$\cC:\cK\to H$ be the linear map induced by the finite-stride convolution in
Definition~\ref{def:convolution}. For every $1\leq q\leq k$, the parametrization
\[
  \Phi=\Pl_H\circ\Gamma_{\cC}:
  \Gr(q,\cK)\longrightarrow\PP(\wedge^qH)
\]
is a closed embedding. If $X=\Gr(q,\cK)$ and $Y=\Phi(X)$, then
$\Phi:X\xrightarrow{\sim}Y$ is an isomorphism. In particular, $\Phi$ is finite
and birational onto its image, $\#\Phi^{-1}(y)=1$ for every $y\in Y$, and
\[
  \dim Y=q(k-q),\qquad \Sing(Y)=\varnothing.
\]
\end{theorem}

The closed-embedding assertion in the main theorem is proved in
Theorem~\ref{thm:Phi-closed}; its finiteness, birationality, uniqueness of
fibers, dimension, and smoothness statements are established in
Corollaries~\ref{cor:finite}--\ref{cor:Y-smooth}.

The family of functions realized by a neural network architecture can be studied
as the image of a parameter-to-function map. In the related literature, such an
image is often called a \emph{neuromanifold}, or a \emph{neural variety} in the
algebraic setting~\cite{APO06,KMMT22,SMK25}. Function spaces, singularities, and
critical points of loss functions for linear convolutional networks have been
studied from an algebraic-geometric perspective~\cite{KMMT22,KMST24}.
Shahverdi, Marchetti, and Kohn further studied polynomial convolutional networks
with monomial activation: their projective parametrization factors through a
Segre--Veronese embedding and is regular and finite birational; they also discuss
the dimension, degree, and singularities of the neural variety and critical
points of a regression loss~\cite[Secs.~3.1, 4.1, and 4.2]{SMK25}. We adopt their
notation for one-dimensional convolution, but study a different object: a
higher-dimensional filter subspace is the single-layer parameter, and we record
the complete Pl\"ucker coordinates of the corresponding operator subspace.

We consider only this single-layer linear subspace of convolution operators and
its complete Pl\"ucker representation. Our theorems make no claims about nonlinear
activations, multilayer compositions, loss functions, or performance in actual
training.

\section{One-Dimensional Convolution and the Grassmann--Pl\texorpdfstring{\"u}{u}cker Parametrization}

\subsection{Finite-stride convolution}

All vector spaces in this paper are finite-dimensional complex vector spaces.
Fix positive integers $k,d,d'$, and $s$, and assume that
\begin{equation}\label{eq:length-relation}
  d=s(d'-1)+k.
\end{equation}
Here $k$ is the filter length, $d$ is the input length, $s$ is the stride, and
$d'$ is the output length. Set
\[
  \cK=\CC^k,\qquad E=\CC^d,\qquad F=\CC^{d'}.
\]
We index coordinates from $0$. This convention agrees with the one-dimensional
valid convolution used in the literature on polynomial convolutional networks
\cite[Sec.~3.1, Eq.~(1)]{SMK25}.

\begin{definition}[Finite-stride convolution]\label{def:convolution}
For a filter
\[
  w=(w_0,\ldots,w_{k-1})\in\cK
\]
and an input
\[
  x=(x_0,\ldots,x_{d-1})\in E,
\]
define $C_wx\in F$ by
\begin{equation}\label{eq:convolution}
  (C_wx)_i=\sum_{j=0}^{k-1}w_jx_{si+j},
  \qquad 0\leq i\leq d'-1.
\end{equation}
By \eqref{eq:length-relation}, if $0\leq i\leq d'-1$ and
$0\leq j\leq k-1$, then
\[
  0\leq si+j\leq s(d'-1)+k-1=d-1,
\]
so every input coordinate in \eqref{eq:convolution} is defined.
\end{definition}

Write
\[
  H=\Hom(E,F).
\]
Convolution induces the map
\begin{equation}\label{eq:C-map}
  \cC:\cK\longrightarrow H,
  \qquad w\longmapsto C_w.
\end{equation}

\begin{proposition}[Linearity and injectivity of the convolution map]
\label{prop:C-injective}
The map $\cC$ defined in \eqref{eq:C-map} is an injective linear map.
\end{proposition}

\begin{proof}
Let $a,b\in\CC$, $w,w'\in\cK$, and $x\in E$. For every
$0\leq i\leq d'-1$,
\[
\begin{aligned}
  \bigl(C_{aw+bw'}x\bigr)_i
  &=\sum_{j=0}^{k-1}(aw_j+bw'_j)x_{si+j}\\
  &=a\sum_{j=0}^{k-1}w_jx_{si+j}
    +b\sum_{j=0}^{k-1}w'_jx_{si+j}\\
  &=\bigl(aC_wx+bC_{w'}x\bigr)_i.
\end{aligned}
\]
Thus $C_{aw+bw'}=aC_w+bC_{w'}$, so $\cC$ is linear.

We next prove injectivity. Let $w\in\ker\cC$, so that $C_w=0$. For every
$0\leq j\leq k-1$, let $e_j\in E$ be the $j$th standard basis vector. By
\eqref{eq:convolution}, the zeroth output coordinate satisfies
\[
  (C_we_j)_0
  =\sum_{\ell=0}^{k-1}w_\ell(e_j)_\ell
  =w_j.
\]
Because $C_w=0$, we have $(C_we_j)_0=0$, and therefore $w_j=0$. This holds for
every $0\leq j\leq k-1$, so $w=0$. Hence
\[
  \ker\cC=\{0\}.
\]
\end{proof}

\begin{remark}\label{rem:no-faithfulness-assumption}
In the concrete convolution model of Definition~\ref{def:convolution}, the
injectivity of $\cC$ is a conclusion of Proposition~\ref{prop:C-injective}, not
an additional assumption. The geometric theorems below are stated for arbitrary
injective linear maps, but when they are applied to the convolution considered
here, the required injectivity has already been proved.
\end{remark}

\subsection{Filter-subspace parameters}

\begin{definition}[Grassmannian]\label{def:grassmannian}
Let $V$ be an $n$-dimensional complex vector space, and let $1\leq q\leq n$.
Write
\[
  \Gr(q,V)=\{U\subseteq V: U\text{ is a }q\text{-dimensional linear subspace}\}.
\]
Endowed with its standard algebraic-variety structure, this is the Grassmannian
of $q$-planes in $V$.
\end{definition}

Fix $1\leq q\leq k$. We take
\[
  X=\Gr(q,\cK)
\]
as the parameter space. A point $U\in X$ represents a $q$-dimensional subspace
of the filter space.

\begin{definition}[Convolution-induced Grassmann map]\label{def:Gamma}
Define
\begin{equation}\label{eq:Gamma}
  \Gamma_{\cC}:\Gr(q,\cK)\longrightarrow\Gr(q,H),
  \qquad U\longmapsto\cC(U).
\end{equation}
\end{definition}

\begin{lemma}[Well-definedness of $\Gamma_{\cC}$]
\label{lem:Gamma-well-defined}
For every $U\in\Gr(q,\cK)$, one has $\dim\cC(U)=q$. Hence
\eqref{eq:Gamma} indeed takes values in $\Gr(q,H)$.
\end{lemma}

\begin{proof}
By Proposition~\ref{prop:C-injective}, the restriction
\[
  \cC|_U:U\longrightarrow H
\]
is injective. The rank--nullity theorem gives
\[
  \dim\cC(U)
  =\dim U-\dim\ker(\cC|_U)
  =q-0=q.
\]
Thus $\cC(U)\in\Gr(q,H)$.
\end{proof}

\subsection{The Pl\texorpdfstring{\"u}{u}cker parametrization}

\begin{definition}[Pl\texorpdfstring{\"u}{u}cker map]\label{def:Plucker}
Let $V$ be a finite-dimensional complex vector space. For
\[
  U=\Span(u_1,\ldots,u_q)\in\Gr(q,V),
\]
define
\begin{equation}\label{eq:Plucker}
  \Pl_V(U)=[u_1\wedge\cdots\wedge u_q]
  \in\PP(\wedge^qV).
\end{equation}
\end{definition}

If $v_1,\ldots,v_q$ is another basis of $U$, then there exists
$A\in\operatorname{GL}_q(\CC)$ such that
\[
  (v_1,\ldots,v_q)=(u_1,\ldots,u_q)A.
\]
Alternating multilinearity of the exterior product gives
\[
  v_1\wedge\cdots\wedge v_q
  =\det(A)\,u_1\wedge\cdots\wedge u_q.
\]
Since $\det(A)\neq0$, the two vectors determine the same projective point, and
thus \eqref{eq:Plucker} is independent of the choice of basis.

\begin{definition}[Grassmann--Pl\texorpdfstring{\"u}{u}cker parametrization and neural variety]
\label{def:Phi}
Define
\begin{equation}\label{eq:Phi}
  \Phi=\Pl_H\circ\Gamma_{\cC}:
  \Gr(q,\cK)\longrightarrow\PP(\wedge^qH).
\end{equation}
Its image
\begin{equation}\label{eq:Y}
  Y=\Phi\bigl(\Gr(q,\cK)\bigr)
\end{equation}
is called the Grassmann--Pl\"ucker neural variety of this single-layer
convolutional model.
\end{definition}

If $u_1,\ldots,u_q$ is a basis of $U$, then
\begin{equation}\label{eq:Phi-basis}
  \Phi(U)
  =\bigl[\cC(u_1)\wedge\cdots\wedge\cC(u_q)\bigr].
\end{equation}

The linear map $\cC$ induces a linear map
\[
  \wedge^q\cC:\wedge^q\cK\longrightarrow\wedge^qH,
\]
whose value on a decomposable vector is
\[
  (\wedge^q\cC)(u_1\wedge\cdots\wedge u_q)
  =\cC(u_1)\wedge\cdots\wedge\cC(u_q).
\]
Because $\cC$ is injective, $\wedge^q\cC$ is also injective and can be
projectivized to give
\[
  \PP(\wedge^q\cC):
  \PP(\wedge^q\cK)\longrightarrow\PP(\wedge^qH).
\]
Equation~\eqref{eq:Phi-basis} gives the commutative identity
\begin{equation}\label{eq:commuting-identity}
  \boxed{
  \Pl_H\circ\Gamma_{\cC}
  =\PP(\wedge^q\cC)\circ\Pl_{\cK}.}
\end{equation}
Thus ``first send the filter subspace to the operator space and then apply the
Pl\"ucker embedding'' and ``first take the Pl\"ucker coordinates of the filter
subspace and then apply the linear map induced on the exterior power'' give the
same parametrization.

\section{Local Coordinates on Grassmannians and the Pl\texorpdfstring{\"u}{u}cker Embedding}

This section recalls the standard facts about Grassmannians, tangent spaces, and
the Pl\"ucker embedding that will be needed below. We give direct proofs of the
local-coordinate and tangent-space statements used later, while citing the
classical Pl\"ucker closed-embedding theorem. For background, see Harris
\cite[Lecture~6, pp.~63--67]{Har92}; for the quotient-space representation,
horizontal tangent spaces, and numerical matrix models of the Grassmann
manifold, see Edelman, Arias, and Smith
\cite[Secs.~2.3.2 and 2.5]{EAS98}.

\subsection{Standard affine coordinates on a Grassmannian}

Let $V$ be an $n$-dimensional complex vector space with fixed ordered basis
$e_1,\ldots,e_n$. For a $q$-element index set
\[
  I=\{i_1<\cdots<i_q\}\subseteq\{1,\ldots,n\},
\]
set
\[
  V_I=\Span(e_{i_1},\ldots,e_{i_q}),
  \qquad
  V_{I^c}=\Span(e_j:j\notin I).
\]
Then $V=V_I\oplus V_{I^c}$. Let
\[
  \pi_I:V\longrightarrow V_I
\]
denote the projection along $V_{I^c}$.

\begin{definition}[Standard Grassmann open set]\label{def:grassmann-chart}
Define
\[
  \cU_I
  =\{U\in\Gr(q,V):\pi_I|_U:U\to V_I\text{ is an isomorphism}\}.
\]
\end{definition}

\begin{proposition}[Graph coordinates]\label{prop:graph-chart}
The map
\[
  \Hom(V_I,V_{I^c})\longrightarrow\cU_I,
  \qquad
  A\longmapsto\operatorname{Graph}(A)
\]
is an isomorphism of affine varieties, where
\[
  \operatorname{Graph}(A)=\{u+A(u):u\in V_I\}.
\]
In particular,
\[
  \cU_I\cong\CC^{q(n-q)}.
\]
\end{proposition}

\begin{proof}
If $A\in\Hom(V_I,V_{I^c})$, then
\[
  \pi_I(u+A(u))=u
\]
for every $u\in V_I$. Hence
\[
  \pi_I|_{\operatorname{Graph}(A)}:
  \operatorname{Graph}(A)\longrightarrow V_I
\]
has inverse $u\mapsto u+A(u)$, and therefore
$\operatorname{Graph}(A)\in\cU_I$.

Conversely, if $U\in\cU_I$, define
\[
  A_U=\pi_{I^c}\circ(\pi_I|_U)^{-1}:V_I\longrightarrow V_{I^c},
\]
where $\pi_{I^c}:V\to V_{I^c}$ is the projection along $V_I$. For $v\in U$,
put $u=\pi_I(v)$. Then
\[
  v=u+\pi_{I^c}(v)=u+A_U(u),
\]
and therefore $U=\operatorname{Graph}(A_U)$. By construction,
$A\mapsto\operatorname{Graph}(A)$ and $U\mapsto A_U$ are inverse maps.

With respect to the fixed bases, the matrix entries of $A$ give $q(n-q)$ affine
coordinates. In these coordinates the two maps above are given, respectively,
by the matrix entries and their identical recovery; hence both are regular.
\end{proof}

\begin{corollary}[Dimension and smoothness]\label{cor:Gr-smooth-dim}
$\Gr(q,V)$ is a smooth projective variety of dimension $q(n-q)$.
\end{corollary}

\begin{proof}
The standard open sets $\cU_I$ cover $\Gr(q,V)$, and by
Proposition~\ref{prop:graph-chart}, each $\cU_I$ is isomorphic to
$\CC^{q(n-q)}$. Thus every point of $\Gr(q,V)$ has an open neighborhood
isomorphic to a smooth affine space, so $\Gr(q,V)$ is smooth of dimension
$q(n-q)$. Its projectivity follows from the Pl\"ucker closed embedding in
Theorem~\ref{thm:Plucker-closed}.
\end{proof}

\begin{lemma}[Irreducibility of the Grassmannian]\label{lem:Gr-irreducible}
$\Gr(q,V)$ is an irreducible algebraic variety.
\end{lemma}

\begin{proof}
Fix $U_0\in\Gr(q,V)$. The algebraic group $\operatorname{GL}(V)$ acts on
$\Gr(q,V)$ by
\[
  (g,U)\longmapsto g(U).
\]
For an arbitrary $U\in\Gr(q,V)$, choose bases of $U_0$ and $U$ and extend each
to a basis of $V$. There is then some $g\in\operatorname{GL}(V)$ with
$g(U_0)=U$. Consequently, the orbit map
\[
  \operatorname{GL}(V)\longrightarrow\Gr(q,V),
  \qquad g\longmapsto g(U_0)
\]
is surjective. The group $\operatorname{GL}(V)$ is the nonempty principal open
subset of the affine space $\operatorname{End}(V)$ defined by $\det\neq0$, so it
is irreducible. The image of an irreducible space under a continuous map is
irreducible, and hence $\Gr(q,V)$ is irreducible.
\end{proof}

\subsection{The tangent space}

\begin{theorem}[Tangent space of a Grassmannian]\label{thm:tangent-Gr}
For every $U\in\Gr(q,V)$, there is a natural linear isomorphism
\begin{equation}\label{eq:tangent-Gr}
  T_U\Gr(q,V)\cong\Hom(U,V/U).
\end{equation}
\end{theorem}

\begin{proof}
Choose a complementary subspace $L\subseteq V$ such that
\[
  V=U\oplus L.
\]
By Proposition~\ref{prop:graph-chart}, an open neighborhood of $U$ is isomorphic
to $\Hom(U,L)$, with $U$ corresponding to the zero map. Therefore
\[
  T_U\Gr(q,V)\cong T_0\Hom(U,L)=\Hom(U,L).
\]
If $\pi:V\to V/U$ is the quotient map, its restriction
\[
  \pi|_L:L\longrightarrow V/U
\]
is a linear isomorphism, and hence it induces
\[
  \Hom(U,L)\xrightarrow{\sim}\Hom(U,V/U),
  \qquad A\longmapsto\pi|_L\circ A.
\]
The composition of these two isomorphisms gives \eqref{eq:tangent-Gr}.

We verify that this isomorphism is independent of the chosen complement. In any
chosen graph coordinates, a tangent vector is represented by a first-order
family of subspaces
\[
  U_t=\{u+t\widetilde A(u):u\in U\},
\]
where $\widetilde A:U\to V$ is linear. Its image in $\Hom(U,V/U)$ is
\[
  u\longmapsto[\widetilde A(u)].
\]
If $\widetilde A'$ determines the same first-order family, then
$\widetilde A(u)-\widetilde A'(u)\in U$, so
\[
  [\widetilde A(u)]=[\widetilde A'(u)]\in V/U.
\]
Thus the resulting element is independent of both the lift and the complement.
\end{proof}

\subsection{Pl\texorpdfstring{\"u}{u}cker coordinates and the classical embedding}

Continue to use the fixed basis $e_1,\ldots,e_n$ of $V$. If a basis of
$U\in\Gr(q,V)$ is arranged as the rows of a full-rank matrix
\[
  M\in\operatorname{Mat}_{q\times n}(\CC),
\]
then for every $q$-element index set $I\subseteq\{1,\ldots,n\}$, define
\begin{equation}\label{eq:Plucker-coordinate}
  p_I(U)=\det(M_I),
\end{equation}
where $M_I$ is the $q\times q$ submatrix formed by the columns indexed by $I$.
Replacing the basis matrix $M$ by $GM$, where
$G\in\operatorname{GL}_q(\CC)$, multiplies every $p_I$ by $\det(G)$. Therefore
\[
  [p_I(U)]_{|I|=q}
\]
is a well-defined system of projective coordinates, and it agrees with the
exterior-product coordinates in Definition~\ref{def:Plucker}.

If an index is repeated, we set the corresponding
$p_{i_1\cdots i_q}$ equal to zero; interchanging two indices changes its sign.
The following classical theorem is used as a cited result and is not reproved
here; see
\cite[Lecture~6, pp.~63--67]{Har92}.

\begin{theorem}[Pl\texorpdfstring{\"u}{u}cker closed embedding]
\label{thm:Plucker-closed}
The Pl\"ucker map
\[
  \Pl_V:\Gr(q,V)\longrightarrow\PP(\wedge^qV)
\]
is a closed embedding. Its image is the projective Grassmann variety, classically
defined by the quadratic Pl\"ucker relations.
\end{theorem}

\section{Regularity of the Grassmann--Pl\texorpdfstring{\"u}{u}cker Parametrization}

\subsection{Meaning of regularity}

\begin{definition}[Regularity used in this paper]
\label{def:regular-parametrization}
Let $f:X\to Z$ be a morphism from a smooth algebraic variety $X$ to an
algebraic variety $Z$. If the differential
\[
  \dd f|_x:T_xX\longrightarrow T_{f(x)}Z
\]
is injective for every $x\in X$, then $f$ is called a \emph{regular
parametrization}. Equivalently,
\[
  \rank(\dd f|_x)=\dim X
\]
for every $x\in X$.
\end{definition}

\begin{remark}
In Definition~\ref{def:regular-parametrization}, ``regular'' means that the
differential has maximal rank everywhere. This differs from the convention in
algebraic geometry in which any morphism of algebraic varieties may be called a
regular map. More precisely, we will prove that $\Phi$ is an immersion
everywhere. We follow the use of ``regular parametrization'' in the literature on
polynomial convolutional networks~\cite[Theorem~4.5]{SMK25}.
\end{remark}

\subsection{The convolution-induced Grassmann map is a morphism}

\begin{proposition}\label{prop:Gamma-morphism}
The map $\Gamma_{\cC}:\Gr(q,\cK)\to\Gr(q,H)$ is a morphism of algebraic
varieties.
\end{proposition}

\begin{proof}
Fix $U_0\in\Gr(q,\cK)$ and choose a complement $L$ such that
\[
  \cK=U_0\oplus L.
\]
Since $\cC$ is injective,
\[
  \cC(\cK)=\cC(U_0)\oplus\cC(L).
\]
Choose a subspace $R\subseteq H$ such that
\[
  H=\cC(U_0)\oplus\cC(L)\oplus R,
\]
and set $Q=\cC(L)\oplus R$. Graph coordinates near $U_0$ in the source
Grassmannian are
\[
  A\in\Hom(U_0,L)
  \longmapsto\operatorname{Graph}(A).
\]
For every such $A$,
\[
\begin{aligned}
  \cC(\operatorname{Graph}(A))
  &=\{\cC(u)+\cC(Au):u\in U_0\}\\
  &=\operatorname{Graph}(B_A),
\end{aligned}
\]
where
\[
  B_A=\cC|_L\circ A\circ(\cC|_{U_0})^{-1}
  \in\Hom(\cC(U_0),Q).
\]
The map $A\mapsto B_A$ is linear and hence regular in graph coordinates. Such
graph-coordinate neighborhoods cover $\Gr(q,\cK)$, so $\Gamma_{\cC}$ is a
morphism.
\end{proof}

By Theorem~\ref{thm:Plucker-closed}, $\Pl_H$ is a morphism. Therefore
$\Phi=\Pl_H\circ\Gamma_{\cC}$ is also a morphism.

\subsection{Differential of the convolution-induced Grassmann map}

Fix $U\in\Gr(q,\cK)$. Define a linear map on quotient spaces by
\begin{equation}\label{eq:Cbar}
  \overline{\cC}_U:\cK/U\longrightarrow H/\cC(U),
  \qquad [v]\longmapsto[\cC(v)].
\end{equation}

\begin{lemma}\label{lem:Cbar-well-defined-injective}
The map $\overline{\cC}_U$ is well defined and injective.
\end{lemma}

\begin{proof}
If $[v]=[v']$, then $v-v'\in U$, and therefore
\[
  \cC(v)-\cC(v')=\cC(v-v')\in\cC(U).
\]
Thus $[\cC(v)]=[\cC(v')]$, which proves well-definedness.

If $\overline{\cC}_U([v])=0$, then $\cC(v)\in\cC(U)$. Hence there exists
$u\in U$ such that
\[
  \cC(v)=\cC(u).
\]
It follows that $\cC(v-u)=0$. Since $\cC$ is injective, $v-u=0$, so $v\in U$
and hence $[v]=0$. Thus $\overline{\cC}_U$ is injective.
\end{proof}

\begin{proposition}[Differential formula]\label{prop:dGamma}
Under the natural isomorphisms
\[
  T_U\Gr(q,\cK)\cong\Hom(U,\cK/U)
\]
and
\[
  T_{\cC(U)}\Gr(q,H)
  \cong\Hom(\cC(U),H/\cC(U)),
\]
one has, for every $A\in\Hom(U,\cK/U)$,
\begin{equation}\label{eq:dGamma}
  \dd\Gamma_{\cC}|_U(A)
  =\overline{\cC}_U\circ A\circ(\cC|_U)^{-1}.
\end{equation}
\end{proposition}

\begin{proof}
Choose a linear lift $\widetilde A:U\to\cK$ such that the quotient map
$\pi_U:\cK\to\cK/U$ satisfies
\[
  \pi_U\circ\widetilde A=A.
\]
By the graph-coordinate description in Theorem~\ref{thm:tangent-Gr}, $A$ is
represented by the first-order family of subspaces
\[
  U_t=\{u+t\widetilde A(u):u\in U\}.
\]
Applying $\Gamma_{\cC}$ gives
\[
\begin{aligned}
  \Gamma_{\cC}(U_t)
  &=\cC(U_t)\\
  &=\{\cC(u)+t\cC(\widetilde A(u)):u\in U\}.
\end{aligned}
\]
Thus the tangent vector in the target Grassmannian maps
$\cC(u)\in\cC(U)$ to
\[
  [\cC(\widetilde A(u))]\in H/\cC(U).
\]
On the other hand,
\[
  A(u)=[\widetilde A(u)]\in\cK/U,
\]
so
\[
  [\cC(\widetilde A(u))]
  =\overline{\cC}_U(A(u)).
\]
Since $\cC|_U:U\to\cC(U)$ is a linear isomorphism,
\[
  \bigl(\dd\Gamma_{\cC}|_U(A)\bigr)(\cC(u))
  =\bigl(\overline{\cC}_U\circ A\circ(\cC|_U)^{-1}\bigr)(\cC(u)),
\]
which proves \eqref{eq:dGamma}.

If $\widetilde A'$ is another lift, then
$(\widetilde A-\widetilde A')(U)\subseteq U$, so
\[
  \cC((\widetilde A-\widetilde A')(U))\subseteq\cC(U).
\]
Therefore the two lifts give the same class in $H/\cC(U)$, and the formula for
the differential is independent of the choice of lift.
\end{proof}

\begin{corollary}\label{cor:dGamma-injective}
For every $U\in\Gr(q,\cK)$, the differential
$\dd\Gamma_{\cC}|_U$ is injective.
\end{corollary}

\begin{proof}
Let $A\in\Hom(U,\cK/U)$ satisfy
\[
  \dd\Gamma_{\cC}|_U(A)=0.
\]
By \eqref{eq:dGamma}, for every $u\in U$,
\[
  \overline{\cC}_U(A(u))=0.
\]
Lemma~\ref{lem:Cbar-well-defined-injective} shows that
$\overline{\cC}_U$ is injective, so $A(u)=0$. This holds for every $u\in U$,
and therefore $A=0$. Hence
\[
  \ker(\dd\Gamma_{\cC}|_U)=\{0\}.
\]
\end{proof}

\begin{remark}[Differential kernel for a noninjective linear map]
\label{rem:general-kernel}
Let $T:V\to H$ be any linear map, put $N=\ker T$, and suppose that
$U\in\Gr(q,V)$ satisfies $U\cap N=\{0\}$. Then $\dim T(U)=q$, and one can
again define $\Gamma_T(U)=T(U)$. The quotient map
\[
  \overline T_U:V/U\longrightarrow H/T(U),
  \qquad[v]\longmapsto[T(v)]
\]
satisfies
\[
\begin{aligned}
  [v]\in\ker\overline T_U
  &\Longleftrightarrow T(v)\in T(U)\\
  &\Longleftrightarrow \text{there exists }u\in U\text{ such that }T(v-u)=0\\
  &\Longleftrightarrow v\in U+N.
\end{aligned}
\]
Consequently,
\[
  \ker\overline T_U=(U+N)/U
\]
and
\[
  \ker(\dd\Gamma_T|_U)
  =\Hom\left(U,\frac{U+N}{U}\right).
\]
For the convolution considered here, $N=\{0\}$, so this formula reduces to
$\ker(\dd\Gamma_{\cC}|_U)=\{0\}$.
\end{remark}

\subsection{Differential of the Pl\texorpdfstring{\"u}{u}cker map}

\begin{proposition}\label{prop:dPlucker-injective}
For every $S\in\Gr(q,H)$, the differential
\[
  \dd\Pl_H|_S:T_S\Gr(q,H)\longrightarrow
  T_{\Pl_H(S)}\PP(\wedge^qH)
\]
is injective.
\end{proposition}

\begin{proof}
By the classical Pl\"ucker closed-embedding theorem,
Theorem~\ref{thm:Plucker-closed} (see
\cite[Lecture~6, pp.~63--67]{Har92}), $\Pl_H$ is a closed immersion.
Every closed immersion is unramified, and an unramified morphism induces an
injective map on Zariski tangent spaces at every point
\cite[Tags~04XV and 0B2G]{Stacks}. Hence $\dd\Pl_H|_S$ is injective.
\end{proof}

\subsection{The regularity theorem}

\begin{theorem}[Regularity of the single-layer parametrization]
\label{thm:regularity}
The differential of
\[
  \Phi:\Gr(q,\cK)\longrightarrow\PP(\wedge^qH)
\]
is injective at every point. More precisely, for every
$U\in\Gr(q,\cK)$,
\[
  \rank(\dd\Phi|_U)=q(k-q).
\]
\end{theorem}

\begin{proof}
By the chain rule,
\[
  \dd\Phi|_U
  =\dd\Pl_H|_{\cC(U)}\circ\dd\Gamma_{\cC}|_U.
\]
Suppose $A\in T_U\Gr(q,\cK)$ satisfies $\dd\Phi|_U(A)=0$. By
Proposition~\ref{prop:dPlucker-injective}, the map
$\dd\Pl_H|_{\cC(U)}$ is injective, and therefore
\[
  \dd\Gamma_{\cC}|_U(A)=0.
\]
Corollary~\ref{cor:dGamma-injective} then shows that
$\dd\Gamma_{\cC}|_U$ is injective, so $A=0$. Hence
\[
  \ker(\dd\Phi|_U)=\{0\}.
\]
By Corollary~\ref{cor:Gr-smooth-dim},
\[
  \dim T_U\Gr(q,\cK)=\dim\Gr(q,\cK)=q(k-q).
\]
It follows that
\[
  \rank(\dd\Phi|_U)=q(k-q).
\]
\end{proof}

\section{The Closed-Embedding Theorem}

\subsection{The Grassmannian isomorphism induced by a linear isomorphism}

Set
\begin{equation}\label{eq:W}
  W=\cC(\cK)\subseteq H.
\end{equation}
By Proposition~\ref{prop:C-injective}, restricting the codomain gives a linear
isomorphism
\[
  \widetilde{\cC}:\cK\xrightarrow{\sim}W.
\]

\begin{lemma}\label{lem:alpha-isomorphism}
The map
\[
  \alpha:\Gr(q,\cK)\longrightarrow\Gr(q,W),
  \qquad U\longmapsto\widetilde{\cC}(U)
\]
is an isomorphism of algebraic varieties, with inverse
\[
  \beta:\Gr(q,W)\longrightarrow\Gr(q,\cK),
  \qquad S\longmapsto\widetilde{\cC}^{-1}(S).
\]
\end{lemma}

\begin{proof}
For every $U\in\Gr(q,\cK)$ and $S\in\Gr(q,W)$,
\[
  \beta(\alpha(U))
  =\widetilde{\cC}^{-1}(\widetilde{\cC}(U))=U
\]
and
\[
  \alpha(\beta(S))
  =\widetilde{\cC}(\widetilde{\cC}^{-1}(S))=S.
\]
Thus $\alpha$ and $\beta$ are inverse set maps.

It remains to verify regularity. Fix a decomposition
\[
  \cK=U_0\oplus L.
\]
Then
\[
  W=\widetilde{\cC}(U_0)\oplus\widetilde{\cC}(L).
\]
In the corresponding graph coordinates, $\alpha$ is
\[
  A\longmapsto
  \widetilde{\cC}|_L\circ A\circ
  (\widetilde{\cC}|_{U_0})^{-1}.
\]
This map is linear and therefore regular. Applying the same calculation to
$\widetilde{\cC}^{-1}$ shows that $\beta$ is also regular. Consequently,
$\alpha$ is an isomorphism of algebraic varieties.
\end{proof}

\subsection{A Pl\texorpdfstring{\"u}{u}cker-coordinate proof for a sub-Grassmannian}

The inclusion of linear subspaces $W\subseteq H$ gives the natural map
\begin{equation}\label{eq:j}
  j:\Gr(q,W)\longrightarrow\Gr(q,H),
  \qquad S\longmapsto S.
\end{equation}

\begin{lemma}[Closed embedding of a sub-Grassmannian]
\label{lem:subgrassmannian-closed}
The map $j$ is a closed embedding.
\end{lemma}

\begin{proof}
Write
\[
  n=\dim W,\qquad m=\dim H.
\]
Choose a basis $e_1,\ldots,e_n$ of $W$ and extend it to a basis
\[
  e_1,\ldots,e_n,e_{n+1},\ldots,e_m
\]
of $H$.

We first prove that $j$ is injective. If $S_1,S_2\in\Gr(q,W)$ and
$j(S_1)=j(S_2)$, then $S_1$ and $S_2$ are equal as linear subspaces of $H$,
and hence $S_1=S_2$. Thus $j$ is injective.

We next describe its image. Define
\begin{equation}\label{eq:Z-subgrassmannian}
  Z=\left\{
  S\in\Gr(q,H):
  p_J(S)=0\text{ for every }J\nsubseteq\{1,\ldots,n\}
  \right\}.
\end{equation}
Here $J\nsubseteq\{1,\ldots,n\}$ means that $J$ contains at least one index
greater than $n$. We prove that
\begin{equation}\label{eq:image-j-Z}
  j(\Gr(q,W))=Z.
\end{equation}

If $S\in\Gr(q,W)$, every basis vector of $S$ belongs to
$\Span(e_1,\ldots,e_n)$, so
\[
  \wedge^qS\subseteq\wedge^qW.
\]
Therefore all Pl\"ucker coordinates containing an index greater than $n$ vanish,
and hence $j(S)\in Z$. This proves
\[
  j(\Gr(q,W))\subseteq Z.
\]

Conversely, take $S\in Z$. The Pl\"ucker coordinates of $S$ are not all zero,
and every coordinate containing an index greater than $n$ is zero. Hence there
exists
\[
  I=\{i_1<\cdots<i_q\}\subseteq\{1,\ldots,n\}
\]
with $p_I(S)\neq0$. Thus $S$ lies in the standard open set $\cU_I$. By
reordering only $e_1,\ldots,e_n$, we may assume that
$I=\{1,\ldots,q\}$. After normalizing $p_I$ to $1$, the subspace $S$ has a
unique row-space matrix
\begin{equation}\label{eq:subgrassmannian-chart-matrix}
  M_S=
  \begin{pmatrix}
    I_q&A&B
  \end{pmatrix},
\end{equation}
where the columns of $A$ correspond to $e_{q+1},\ldots,e_n$, and the columns of
$B$ correspond to $e_{n+1},\ldots,e_m$.

Let $b_{r\ell}$ be the entry of $B$ in row $r$ and in the column corresponding
to $e_\ell$, where $1\leq r\leq q$ and $n+1\leq\ell\leq m$. Computing the
corresponding minor of \eqref{eq:subgrassmannian-chart-matrix} gives
\[
  p_{\{1,\ldots,\widehat r,\ldots,q,\ell\}}(S)
  =(-1)^{q-r}b_{r\ell}p_{\{1,\ldots,q\}}(S)
  =(-1)^{q-r}b_{r\ell}.
\]
This index set contains $\ell>n$. Since $S\in Z$, the left-hand side is zero,
and therefore
\[
  b_{r\ell}=0
\]
for every $r,\ell$. Thus $B=0$, every row of $M_S$ belongs to $W$, and
$S\subseteq W$. Hence $S\in j(\Gr(q,W))$, proving
\eqref{eq:image-j-Z}.

Under the Pl\"ucker embedding, \eqref{eq:Z-subgrassmannian} can be written as
\begin{equation}\label{eq:intersection-subgrassmannian}
  \Pl_H(Z)
  =\Pl_H(\Gr(q,H))\cap\PP(\wedge^qW).
\end{equation}
The space $\PP(\wedge^qW)$ is the projective linear subspace of
$\PP(\wedge^qH)$ defined by the homogeneous linear equations
\[
  p_J=0,
  \qquad J\nsubseteq\{1,\ldots,n\}.
\]
By Theorem~\ref{thm:Plucker-closed}, $\Pl_H(\Gr(q,H))$ is a closed subvariety.
Therefore \eqref{eq:intersection-subgrassmannian} shows that $Z$ is a closed
subvariety of $\Gr(q,H)$.

Finally, we prove that $j$ identifies $\Gr(q,W)$ isomorphically with $Z$. On any
standard open set with $I\subseteq\{1,\ldots,n\}$, points of $\Gr(q,W)$ are
represented by matrices
\[
  \begin{pmatrix}I_q&A\end{pmatrix},
\]
whereas the corresponding points of $Z$ are represented by
\[
  \begin{pmatrix}I_q&A&0\end{pmatrix}.
\]
Thus, in these affine coordinates, $j$ is
\[
  A\longmapsto(A,0),
\]
and its inverse on the image is
\[
  (A,0)\longmapsto A.
\]
Both maps are given by coordinate polynomials and are therefore regular. These
standard open sets cover $\Gr(q,W)$ and $Z$, so
\[
  j:\Gr(q,W)\xrightarrow{\sim}Z
\]
is an isomorphism. Since $Z$ is closed in $\Gr(q,H)$, the map $j$ is a closed
embedding.
\end{proof}

\begin{remark}
The proof of Lemma~\ref{lem:subgrassmannian-closed} uses only the injectivity of
$j$, the vanishing equations for Pl\"ucker coordinates, and the standard affine
coordinates on Grassmannians. On each standard affine open set, the
closed-embedding property is verified directly by the coordinate map
$A\mapsto(A,0)$.
\end{remark}

\subsection{The main closed-embedding theorem}

\begin{theorem}[Closed embedding of the convolution-induced Grassmann map]
\label{thm:Gamma-closed}
The map
\[
  \Gamma_{\cC}:\Gr(q,\cK)\longrightarrow\Gr(q,H)
\]
is a closed embedding.
\end{theorem}

\begin{proof}
By Lemma~\ref{lem:alpha-isomorphism},
\[
  \alpha:\Gr(q,\cK)\xrightarrow{\sim}\Gr(q,W)
\]
is an isomorphism. By Lemma~\ref{lem:subgrassmannian-closed},
\[
  j:\Gr(q,W)\hookrightarrow\Gr(q,H)
\]
is a closed embedding. For every $U\in\Gr(q,\cK)$,
\[
  (j\circ\alpha)(U)=j(\cC(U))=\cC(U)=\Gamma_{\cC}(U).
\]
Thus
\[
  \Gamma_{\cC}=j\circ\alpha.
\]
The composition of an isomorphism with a closed embedding is a closed embedding,
so $\Gamma_{\cC}$ is a closed embedding.
\end{proof}

\begin{theorem}[Closed embedding of the Grassmann--Pl\texorpdfstring{\"u}{u}cker parametrization]
\label{thm:Phi-closed}
The parametrization
\[
  \Phi:\Gr(q,\cK)\longrightarrow\PP(\wedge^qH)
\]
is a closed embedding.
\end{theorem}

\begin{proof}
By Theorem~\ref{thm:Gamma-closed}, $\Gamma_{\cC}$ is a closed embedding; by
Theorem~\ref{thm:Plucker-closed}, $\Pl_H$ is a closed embedding. Therefore their
composition
\[
  \Phi=\Pl_H\circ\Gamma_{\cC}
\]
is a closed embedding.
\end{proof}

\begin{remark}[Relation to the regularity theorem]
Theorem~\ref{thm:Phi-closed} also implies that the differential of $\Phi$ is
injective everywhere. This implication was not used to prove
Theorem~\ref{thm:regularity}. The proof of regularity uses the direct
differential formula \eqref{eq:dGamma} for the convolution-induced Grassmann map
together with the classical closed-embedding theorem for the Pl\"ucker map. It
does not use the closed-embedding property of the full parametrization $\Phi$,
which is established only in Theorem~\ref{thm:Phi-closed}. Therefore the
regularity argument and the proof of the main closed-embedding theorem are not
circular.
\end{remark}

\section{Finite Birationality and Geometric Consequences}

Continue to write
\[
  X=\Gr(q,\cK),
  \qquad
  Y=\Phi(X)\subseteq\PP(\wedge^qH).
\]

\begin{corollary}[Isomorphism between the parameter space and the neural variety]
\label{cor:X-isomorphic-Y}
The set $Y$ is a closed subvariety of $\PP(\wedge^qH)$, and
\[
  \Phi:X\xrightarrow{\sim}Y
\]
is an isomorphism of algebraic varieties.
\end{corollary}

\begin{proof}
By definition, a closed embedding identifies $X$ isomorphically with a closed
subvariety of the target. That closed subvariety is exactly the image $Y$ of
$\Phi$.
\end{proof}

\begin{corollary}[Finiteness]\label{cor:finite}
The morphism
\[
  \Phi:X\longrightarrow\PP(\wedge^qH)
\]
is finite. In particular, $\Phi:X\to Y$ is finite.
\end{corollary}

\begin{proof}
A closed embedding is a finite morphism~\cite[Tag~035C]{Stacks}. More explicitly,
let $V=\operatorname{Spec}A$ be an affine open subset of the target. Since
$\Phi$ is a closed embedding, there is an ideal $I\subseteq A$ such that
\[
  \Phi^{-1}(V)=\operatorname{Spec}(A/I).
\]
As an $A$-module, $A/I$ is generated by $1+I$, and hence is finitely generated.
Therefore $\Phi$ is finite.
\end{proof}

\begin{corollary}[Birationality]\label{cor:birational}
The morphism $\Phi:X\to Y$ is birational.
\end{corollary}

\begin{proof}
By Lemma~\ref{lem:Gr-irreducible}, $X$ is irreducible. By
Corollary~\ref{cor:X-isomorphic-Y}, $Y$ is isomorphic to $X$ and is therefore
also irreducible. The same corollary gives an isomorphism of function fields
\[
  \Phi^*:\CC(Y)\xrightarrow{\sim}\CC(X).
\]
Thus $\Phi:X\to Y$ is birational. In fact, it is stronger than a birational map:
it is an isomorphism on all of $X$, not merely on a dense open subset.
\end{proof}

\begin{corollary}[Uniqueness of fibers]\label{cor:fibers}
For every $y\in Y$,
\[
  \#\Phi^{-1}(y)=1.
\]
\end{corollary}

\begin{proof}
By Corollary~\ref{cor:X-isomorphic-Y}, $\Phi:X\to Y$ is an isomorphism, so its
underlying map of sets is bijective. Hence every $y\in Y$ has exactly one
preimage.
\end{proof}

\begin{corollary}[Dimension and smoothness]\label{cor:Y-smooth}
The neural variety $Y$ satisfies
\[
  \dim Y=q(k-q)
\]
and
\[
  \Sing(Y)=\varnothing.
\]
\end{corollary}

\begin{proof}
By Corollary~\ref{cor:X-isomorphic-Y}, $Y\cong X=\Gr(q,\cK)$. Isomorphisms
preserve local rings and therefore preserve dimension and the property of a
local ring being regular. By Corollary~\ref{cor:Gr-smooth-dim},
\[
  \dim X=q(\dim\cK-q)=q(k-q)
\]
and $X$ is smooth. Therefore $Y$ has the same dimension and is smooth; that is,
$\Sing(Y)=\varnothing$.
\end{proof}

\begin{corollary}[The rank-one case]\label{cor:q-one}
If $q=1$, then
\[
  \Gr(1,\cK)=\PP(\cK),
\]
and the parametrization reduces to the projective linear embedding
\[
  \Phi:\PP(\cK)\longrightarrow\PP(H),
  \qquad[w]\longmapsto[\cC(w)].
\]
\end{corollary}

\begin{proof}
By definition, $\Gr(1,\cK)$ is the set of one-dimensional subspaces of $\cK$,
which is $\PP(\cK)$. Moreover, $\wedge^1H=H$, and
$\Pl_H:\Gr(1,H)\to\PP(H)$ is the identity identification. Thus
\eqref{eq:Phi-basis} becomes
\[
  \Phi([w])=[\cC(w)].
\]
Since $\cC$ is injective, its projectivization is a projective linear closed
embedding.
\end{proof}

\subsection{A concrete nontrivial symbolic computation}
\label{subsec:symbolic-example}

To check the preceding conclusions in a completely reproducible
low-dimensional case, take
\[
  k=4,\qquad q=2,\qquad s=1,\qquad d'=2,\qquad d=5.
\]
For $w=(w_0,w_1,w_2,w_3)\in\cK=\CC^4$,
Definition~\ref{def:convolution} gives
\begin{equation}\label{eq:example-convolution-matrix}
  C_w=
  \begin{pmatrix}
    w_0&w_1&w_2&w_3&0\\
    0&w_0&w_1&w_2&w_3
  \end{pmatrix}
  \in H=\Hom(\CC^5,\CC^2)\cong\CC^{10}.
\end{equation}
Let $E_1,\ldots,E_{10}$ be the standard basis of $H$ in row-major order,
and let $e_1,\ldots,e_4$ be the standard basis of $\cK$. Then
\begin{equation}\label{eq:example-C-basis}
  \cC(e_i)=E_i+E_{i+6},\qquad 1\leq i\leq4.
\end{equation}

Let $p_{ij}$, $1\leq i<j\leq4$, be the source Pl\"ucker coordinates on
$\Gr(2,4)$. They satisfy
\begin{equation}\label{eq:example-source-Plucker}
  p_{12}p_{34}-p_{13}p_{24}+p_{14}p_{23}=0.
\end{equation}
Write $z_{ab}$, $1\leq a<b\leq10$, for the homogeneous coordinates on
$\PP(\wedge^2H)=\PP^{44}$. By \eqref{eq:example-C-basis}, for every
$1\leq i<j\leq4$,
\begin{align}\label{eq:example-wedge-map}
  (\wedge^2\cC)(e_i\wedge e_j)
  &=(E_i+E_{i+6})\wedge(E_j+E_{j+6})\notag\\
  &=E_i\wedge E_j+E_i\wedge E_{j+6}
    -E_j\wedge E_{i+6}+E_{i+6}\wedge E_{j+6}.
\end{align}
Hence the coordinates of the image satisfy
\begin{equation}\label{eq:example-linear-relations}
  z_{ij}=z_{i,j+6}=z_{i+6,j+6}=p_{ij},
  \qquad z_{j,i+6}=-p_{ij},
\end{equation}
and the remaining $21$ coordinates $z_{ab}$ not occurring in
\eqref{eq:example-linear-relations} vanish.

Over the rational field, we used Singular~4.4.1~\cite{Singular441} to form
the graph ideal defined by
\eqref{eq:example-source-Plucker}--\eqref{eq:example-linear-relations} and
then eliminated the six source coordinates $p_{ij}$. The resulting elimination
ideal $I_Y$ is the homogeneous ideal generated by
\begin{enumerate}[label=\textnormal{(\roman*)},leftmargin=2.8em]
  \item the $21$ vanishing coordinates $z_{ab}$ described above;
  \item the three independent linear relations in
        \eqref{eq:example-linear-relations} for each
        $1\leq i<j\leq4$, giving $18$ linear relations in total; and
  \item the single quadratic relation
  \begin{equation}\label{eq:example-target-Plucker}
    z_{12}z_{34}-z_{13}z_{24}+z_{14}z_{23}=0.
  \end{equation}
\end{enumerate}
More precisely, reducing the generators of the elimination ideal by the
standard basis of the expected ideal, and conversely, gives the zero ideal in
both directions; the reduced Gr\"obner basis has $40$ elements. Thus the
calculation recovers the homogeneous image ideal, not only its set of points.

The Hilbert series of the homogeneous coordinate ring simplifies to
\begin{equation}\label{eq:example-Hilbert-series}
  \operatorname{Hilb}_{\CC[z]/I_Y}(t)
  =\frac{1+t}{(1-t)^5}.
\end{equation}
Consequently,
\[
  \dim Y=4=2(4-2),\qquad \deg Y=2.
\]
On the standard chart $p_{12}\neq0$, write
\[
  U=\operatorname{rowspan}
  \begin{pmatrix}
    1&0&a&b\\
    0&1&c&d
  \end{pmatrix}.
\]
Then
\[
  (z_{13},z_{14},z_{23},z_{24})=(c,d,-a,-b),
\]
and therefore
\begin{equation}\label{eq:example-chart-Jacobian}
  \det\frac{\partial(z_{13},z_{14},z_{23},z_{24})}
  {\partial(a,b,c,d)}=1.
\end{equation}
This directly checks $\rank(\dd\Phi)=4$ on this chart. Finally, in the six
independent coordinates, the image is defined only by
\eqref{eq:example-target-Plucker}. All first partial derivatives of this
quadratic vanish simultaneously only when the six coordinates are zero.
Thus the affine cone is singular only at its vertex, and the corresponding
projective image is smooth.

This computation therefore recovers the image ideal and checks the dimension,
full chart rank, and smoothness predicted by the main theorem in the first
concrete convolutional case with a nontrivial Pl\"ucker relation. It is a
low-dimensional symbolic illustration, not a computer proof of the general
theorem. The reproducible script is supplied as
\texttt{convolution\_q2\_k4.sing}.

\section{Applications and Potential Value}

\subsection{Status of the statements in this section}

The results proved above are that $\Phi$ has injective differential everywhere,
$\Phi$ is a closed embedding, $X\cong Y$, and $Y$ is a smooth projective
variety. This section discusses modeling directions that these results may
support, but does not assert the following as consequences of the theorems in
this paper:
\[
  \begin{gathered}
  \text{higher predictive accuracy, better generalization, shorter actual
  running time,}\\
  \text{or greater empirical robustness}.
  \end{gathered}
\]
These properties depend on the particular network, data, optimization algorithm,
and implementation, and must be validated by additional theoretical analysis or
numerical experiments.

\subsection{Change-of-basis redundancy and intrinsic degrees of freedom}

Suppose that a convolutional layer contains $m$ filters
\[
  w_1,\ldots,w_m\in\cK,
\]
and let
\[
  U=\Span(w_1,\ldots,w_m),
  \qquad \dim U=q<m.
\]
Choose a basis $u_1,\ldots,u_q$ of $U$. For every $A\in\GL(q)$, the family
\[
  (v_1,\ldots,v_q)=(u_1,\ldots,u_q)A
\]
is again a basis of $U$. Thus the right action of $\GL(q)$ on a basis matrix does
not change the Grassmann point.

Let $n=\dim\cK$. A full-rank $n\times q$ basis matrix has $nq$ coordinates,
whereas
\[
  \dim\GL(q)=q^2.
\]
Removing the change-of-basis freedom leaves
\[
  nq-q^2=q(n-q)=\dim\Gr(q,\cK).
\]
Thus the Grassmann parametrization removes the intrinsic redundancy caused by a
choice of basis rather than deleting coordinates arbitrarily. This quotient-space
viewpoint is the standard interpretation of the Grassmannian
\cite[Lecture~6, pp.~63--66]{Har92}.

It is important that the point $U$ records only the subspace spanned by the
filters. If the task requires recovering every $w_i$, their coefficients in a
chosen basis must also be stored. The closed-embedding theorem states only that
the complete Pl\"ucker representation uniquely recovers $U$; it does not state
that the discarded ordered family of filters can be recovered.

\subsection{Low-rank filter families and convolutional compression}

Suppose that the filters are approximately contained in a $q$-dimensional
subspace
\[
  U=\Span(u_1,\ldots,u_q),
\]
in the sense that there are coefficients $a_{ij}$ with
\[
  w_i\approx\sum_{j=1}^q a_{ij}u_j.
\]
By linearity of convolution in the filter,
\[
  C_{w_i}
  \approx\sum_{j=1}^q a_{ij}C_{u_j}.
\]
Thus one may first compute $q$ basic convolutional responses and then form the
$m$ approximate responses by linear combinations.

If $n=\dim\cK$, directly storing $m$ filters requires $mn$ scalars. Storing a
family of $q$ basis vectors and an $m\times q$ coefficient matrix requires
\[
  qn+mq=q(n+m)
\]
scalars. Under this elementary count, the latter representation reduces the
number of parameters only if
\begin{equation}\label{eq:compression-threshold}
  q(n+m)<mn,
  \qquad\text{that is,}\qquad
  q<\frac{mn}{m+n}.
\end{equation}
Therefore the use of Grassmann parameters does not automatically yield
compression. One must also prove or observe that the filter family has
sufficiently low effective rank and account for the costs of basis
orthogonalization, coefficient mixing, and the storage format.

Previous work has used channel or filter redundancy in convolutional filters to
construct low-rank separable approximations that accelerate pretrained
convolutional networks~\cite{DZBLF14,JVZ14}. Those works decompose a particular
convolution tensor and report the corresponding experiments, whereas this paper
records the subspace spanned by the filters. The connection is a modeling
motivation, not an equivalence of models or algorithms.

\section{Future Work}

\subsection{Multilayer Grassmann parameter spaces}

If the filter space in layer $i$ is $\cK_i$ and the subspace dimension is $q_i$,
then a formal multilayer parameter space is
\[
  X_L=\prod_{i=0}^{L-1}\Gr(q_i,\cK_i),
\]
with dimension
\[
  \dim X_L
  =\sum_{i=0}^{L-1}q_i(\dim\cK_i-q_i).
\]
If every $q_i=1$, each factor reduces to a projective filter space. This product
alone, however, does not define a function parametrization for a deep network.
The basis-independent output of the first layer naturally carries a factor
$U_0^\vee$, and consequently the input representation of the next layer depends
on the parameters of the preceding layer. A composition between layers that is
compatible with changes of basis in every layer must first be defined before one
can discuss the differential, fibers, and closed-embedding property of the total
map. The single-layer theorem cannot simply be applied layer by layer.

Deep models in which Grassmann data serve as layer inputs or representations
have been constructed using full-rank mappings, reorthogonalization, projection
pooling, and manifold backpropagation~\cite{HWVG18}. In those models a Grassmann
point represents data or an activation subspace, whereas in the present paper a
Grassmann point parametrizes a subspace of convolutional filters. The two
settings cannot be directly identified, but the former provides techniques that
may inform the numerical design of layers.

\subsection{Nonlinear activation and change-of-basis equivariance}

After a basis of $U$ has been chosen, let the channel vector be $z\in\CC^q$. A
change of basis replaces its coordinates by $Gz$, where $G\in\GL(q)$. For the
coordinatewise power activation
\[
  \sigma_r(z_1,\ldots,z_q)=(z_1^r,\ldots,z_q^r),
\]
one generally has
\[
  \sigma_r(Gz)\neq G\sigma_r(z).
\]
For example, if $r=2$ and
\[
  z=\begin{pmatrix}1\\1\end{pmatrix},
  \qquad
  G=\begin{pmatrix}1&1\\0&1\end{pmatrix},
\]
then
\[
  \sigma_2(Gz)=\begin{pmatrix}4\\1\end{pmatrix},
  \qquad
  G\sigma_2(z)=\begin{pmatrix}2\\1\end{pmatrix}.
\]
Thus an ordinary coordinatewise activation depends on the chosen basis and does
not automatically descend to a map that depends only on the Grassmann point.

One candidate intrinsic construction is the symmetric-tensor map
\[
  \nu_r:W\longrightarrow\Sym^r(W),
  \qquad z\longmapsto z^{\odot r},
\]
because, for every linear map $G:W\to W'$,
\[
  \Sym^r(G)(z^{\odot r})=(Gz)^{\odot r}.
\]
This construction, however, would make later layers act on parameter-dependent
symmetric-tensor spaces, so the total parametrization would have to be rebuilt.
Equivariant networks for general matrix groups require both linear and nonlinear
layers to accommodate the relevant group representations and may use gated or
tensor-product nonlinearities~\cite{FWW21}. These methods provide only design
directions and do not themselves prove a multilayer extension of the present
model.

\section{Limitations}

The results of this paper have the following explicit boundaries.

First, we study the complete Pl\"ucker representation of an operator subspace,
not the full network function class obtained after arbitrary activation,
readout, and classification layers. If the readout does not have the appropriate
$\GL(q)$ invariance or equivariance, the output will depend on the choice of
basis.

Second, the proof of injectivity concerns the full input space and the valid
convolution of Definition~\ref{def:convolution}. If the boundary conditions are
changed, output coordinates are deleted, the input class is restricted, or
redundant filter parameters are used, one must prove $\ker\cC=0$ again.

Third, the theory is formulated over complex algebraic varieties. Practical
optimization is usually carried out on real Grassmann manifolds. Although the
linear-algebra formulas can be restricted to the real field, the topology of the
real points, numerical stability, and optimization dynamics are not direct
consequences of a complex-algebraic closed-embedding result.

Fourth, the ambient vector space for the Pl\"ucker representation has dimension
\[
  \dim\wedge^qH=\binom{\dim H}{q}.
\]
Even using $Y\subseteq\PP(\wedge^qW)$ and $\dim W=k$, the number of coordinates
can still be $\binom{k}{q}$. Pl\"ucker coordinates are therefore well suited to
theoretical analysis, but in large-scale numerical computation they may be less
economical than orthonormal basis matrices or projection matrices.

Fifth, we fix the subspace dimension $q$ and do not treat changes in effective
rank during training or prove how to select $q$ automatically from data. We
provide only one low-dimensional symbolic computation and no systematic
training or data experiments. We therefore do not claim that the model has
already improved accuracy, generalization, computational complexity, or
robustness.

\section{Conclusion}

Starting from the concrete one-dimensional finite-stride convolution
\[
  (C_wx)_i=\sum_{j=0}^{k-1}w_jx_{si+j},
\]
we proved that the linear filter-to-convolution-operator map
\[
  \cC:\cK\longrightarrow H
\]
is injective. Consequently, for every $q$-dimensional filter subspace $U$, the
operator image $\cC(U)$ is again $q$-dimensional, giving the parametrization
\[
  \Phi=\Pl_H\circ\Gamma_{\cC}.
\]

Locally, using the natural isomorphism
\[
  T_U\Gr(q,\cK)\cong\Hom(U,\cK/U)
\]
and the differential formula
\[
  \dd\Gamma_{\cC}|_U(A)
  =\overline{\cC}_U\circ A\circ(\cC|_U)^{-1},
\]
we proved that $\dd\Gamma_{\cC}|_U$ is injective. The classical Pl\"ucker
closed-embedding theorem, together with the general fact that closed immersions
induce injective tangent maps, gives the injectivity of $\dd\Pl_H$, and hence
\[
  \rank(\dd\Phi|_U)=q(k-q).
\]

Globally, let $W=\cC(\cK)$. The linear isomorphism $\cK\cong W$ induces
\[
  \Gr(q,\cK)\cong\Gr(q,W).
\]
Using Pl\"ucker coordinates, we proved that the image of
\[
  \Gr(q,W)\hookrightarrow\Gr(q,H)
\]
is defined exactly by the vanishing of all Pl\"ucker coordinates containing an
index external to $W$, and that the map is $A\mapsto(A,0)$ in standard affine
coordinates. It is therefore a closed embedding, and consequently $\Phi$ is a
closed embedding.

The resulting chain of strict implications is
\[
  \begin{gathered}
  \cC\text{ is injective}
  \Longrightarrow
  \Phi\text{ is a closed embedding}
  \Longrightarrow
  \Phi:X\xrightarrow{\sim}Y\\
  \Longrightarrow
  \begin{cases}
    \Phi\text{ is finite and birational onto its image},\\
    \#\Phi^{-1}(y)=1,\\
    \dim Y=q(k-q),\\
    \Sing(Y)=\varnothing.
  \end{cases}
  \end{gathered}
\]

These conclusions apply to the single-layer linear subspace of convolution
operators defined in this paper and to its complete Pl\"ucker representation.
If nonlinear activation, multilayer composition, output projection, or
incomplete coordinate observation is introduced, then well-definedness of the
parametrization, injectivity of the convolution map, and the differential and
fiber structures must be checked again. The closed-embedding result in this
paper cannot be transferred to such models without new proofs.

As a concrete check of the general conclusions, we also performed a Singular
elimination computation for $k=4,q=2,d=5,d'=2,s=1$. It directly recovered the
image ideal generated by $21$ coordinate vanishings, $18$ linear
identifications, and one Klein quadratic, and yielded $\dim Y=4$,
$\deg Y=2$, full differential rank on the standard chart, and
$\Sing(Y)=\varnothing$, in complete agreement with the main theorem.

\section*{Declaration on the Use of Generative Artificial Intelligence}

In accordance with the current SIAM editorial policy on artificial
intelligence~\cite{SIAMAI26}, the authors make the following declaration.
Generative artificial intelligence tools, including OpenAI ChatGPT and Codex,
were used to assist with the organization and linguistic revision of the
manuscript, the drafting and revision of mathematical exposition and arguments,
literature-search support, bibliographic preparation, and \LaTeX{} formatting.
All mathematical statements, proofs, citations, and bibliography entries
included in the final manuscript were independently reviewed and verified by the
authors. The authors assume responsibility for all content.

\section{Acknowledgement}
	
	On behalf of all authors, the corresponding author states that there is no conflict of interest.
H. Zuo acknowledges support from NSFC (grant No. 12671056) and BJNSF (grant No. 1252009).

\end{document}